\documentclass{amsart}

\newtheorem{thm}{Theorem}[section]
\newtheorem{lem}[thm]{Lemma}

\newtheorem{main}{Theorem}

\newcommand{\aut}[1]{\textnormal{Aut}(#1)}

\begin{document}

\title[Enhanced power graphs with universal vertices]{Characterizing finite groups whose enhanced power graphs have universal vertices}

\author[D.G.\ Costanzo]{David G.\ Costanzo}
\address{School of Mathematical and Statistical Sciences, O-110 Martin Hall, Box 340975, Clemson University, Clemson, SC 29634, USA}
\email{davidgcostanzo@gmail.com}

\author[M.L.\ Lewis]{Mark L.\ Lewis}
\address{Department of Mathematical Sciences, 1300 Lefton Esplanade, Kent State University, Kent, OH 44242, USA}
\email{lewis@math.kent.edu}

\author[S.\ Schmidt]{Stefano Schmidt}

\author[E.\ Tsegaye]{Eyob Tsegaye}
\address{Department of Mathematics, Princeton University, Princeton, NJ 08540}
\email{eyob@princeton.edu}

\author[G.\ Udell]{Gabe Udell}
\address{Department of Mathematics, Malott Hall, 212 Garden Ave, Cornell University, Ithaca, NY 14853}
\email{gru5@cornell.edu}



\begin{abstract}
Let $G$ be a finite group and construct a graph $\Delta(G)$ by taking $G\setminus\{1\}$ as the vertex set of $\Delta(G)$ and by drawing an edge between two vertices $x$ and $y$ if $\langle x,y\rangle$ is cyclic.
Let $K(G)$ be the set consisting of the universal vertices of $\Delta(G)$ along the identity element.
For a solvable group $G$, we present a necessary and sufficient conditon for $K(G)$ to be nontrivial.
We also develop a connection between $\Delta(G)$ and $K(G)$ when $|G|$ is divisible by two distinct primes and the diameter of $\Delta(G)$ is $2$.
\end{abstract}

\maketitle

\section{Introduction}

In this paper, all groups are finite. Let $G$ be a finite group. The \textit{enhanced power graph} of $G$ is the graph whose vertex set is $G \setminus \{ 1 \}$ and there is an edge between $x,y \in G \setminus \{ 1\}$ if $\langle x, y \rangle$ is cyclic. The enhanced power graph of $G$ will be denoted by $\Delta(G)$. This graph was first studied in 2008 under the name ``cyclic graph'' in \cite{imper}.  That paper was motivated by the paper \cite{AbdHas1} where they essentially considered the complement graph: in particular, they considered the graph, which they called the noncyclic graph, where the vertices of this graph is the set $G$ and there is an edge between $x$ and $y$ whenever $\langle x,y \rangle$ is noncyclic.  

It appears that the name enhanced power graph was introduced in the paper \cite{AACNS}. Apparently, the authors of that paper were unaware that the graph had been previously studied under the name cyclic graph in \cite{imper} and the followup paper \cite{imle}.  We recommend the reader consult the excellent expository paper \cite{camex} for an account of graphs and groups.  Since that paper has been published, there has been an explosion in the number of papers published on the enhanced power graph, and we are not going to try to reference all of those papers.  Also, there has been an expository paper that focuses just on enhanced power graphs of groups and it is worth consulting this paper \cite{enhex}.

In graph theory, a \textit{universal vertex} is a vertex that is adjacent to all the other vertices in the graph.  Said differently, if a graph has $n$ vertices, then a universal vertex is a vertex of degree $n -1$. We note that we have seen a number of different names used for such vertices in the literature. In particular, we have seen these vertices called complete vertices, cone vertices, and dominating vertices.  We note that dominating vertex seems to be favored among the literature for enhanced power graph papers and in fact was the term we use in \cite{Zgps}.  However, universal vertex seems to be the term favored in graph theory, and so, we have decided to use that term here.

We define $K(G)$ to be the set of all universal vertices in the enhanced power graph of $G$, along with the identity of $G$.  We discuss the background for $K(G)$ in Section \ref{K(G)}.  In this paper, we will characterize the groups where $K(G) > 1$ and then we will see that when $G$ is a $\{ p,q\}$-group, we will be able to determine when $K (G) > 1$ by looking at $\Delta (G)$.  We note that we are not able to say much when $p$ is odd or when $p =2$ and a Sylow $2$-subgroup is cyclic.  We see that when $p = 2$ and a Sylow $2$-subgroup is not cyclic, the structure of $G$ is rather limited.  Our first main theorem characterizes the situation where $K (G) > 1$.  We will define the notation used in this theorem in Section \ref{K(G)}.

\begin{main} \label{K(G) main}
Let $G$ be a solvable group, let $p$ be a prime, and let $P$ be a Sylow $p$-subgroup of $G$. Then $p$ is a prime divisor of $|K(G)|$ if and only if one of following situations occurs:
\begin{enumerate}
\item $P$ is cyclic and $(Z (G) \cap P) \neq 1$.
\item The prime $p=2$ and one of the following situations occurs.
	\begin{enumerate}
	\item $P$ is generalized quaternion, $G$ has a normal $2$-complement $Q$, and $Z(P) \le C_P (Q)$.
	\item $Z(P) \le C_P (O_{2'}(G))$ and $G/O_{2'}(G)$ is isomorphic to either $\textnormal{SL}_{2}(3)$ or $\widetilde{\textnormal{GL}}_{2} (3)$.  If in addition $3$ divides $|K(G)|$, then $G/O_{2'}(G)\cong\textnormal{SL}_{2} (3)$.
	\end{enumerate}
\end{enumerate}
If conclusion (2) occurs,  then $|K(G)|_{2}=2$.
\end{main}

We now turn to the case where $G$ is a $\{p,q\}$-group for distinct primes $p$ and $q$.  We prove the following theorem:

\begin{main}\label{Main 2}
Let $p$ and $q$ be distinct primes, and let $G$ be a $\{p, q\}$-group.	Then $\textnormal{diam}(\Delta(G))=2$ if and only if $|K(G)| > 1$.
\end{main}

We would like to conclude this introduction with an open question.  In particular, we wonder how much Theorem \ref{Main 2} can be extended beyond groups that are divisible by two primes.  While it seems too much to hope that for a general group $G$ or even a general solvable group $G$ that the enhanced power graph of $G$ has diameter $2$ if and only if $K (G) > 1$.  Recall that $K(G) \le Z(G)$.   We wonder if it might be true that $\textnormal{diam}(\Delta(G))=2$ if and only if $|Z(G)| > 1$.   The meager evidence that we have been able to gather suggests that this should be true at least when $|G|$ is divisible by three primes.

This research was conducted during a summer REU in 2020 at Kent State University with the funding of NSF Grant DMS-1653002. We thank the NSF and Professor Soprunova for their support.

\section{Solvable groups with $K(G)$ nontrivial}\label{K(G)}

Let $G$ be a group, and define \[K(G) = \{ x \in G \mid \langle x,g \rangle\: \textnormal{is cyclic for all}\: g \in G \}.\]  Hence, $K(G)$ is the set of universal vertices for the enhanced power graph of $G$ along with the identity element of $G$. This set seems to have been originally studied independent of the cyclic graph in \cite{cycels}. In Theorem 1 of \cite{cycels}, the authors prove that $K(G)$ is a subgroup of $G$, and they note that it is contained in the center of $G$.  In Theorem 2 of that paper, they prove that $K(G)$ is the intersection of the maximal cyclic subgroups of $G$, and in Theorem 4 of that paper, they prove that $K(G/K(G)) = 1$.  The intersection of the maximal cyclic subgroups of $G$ is also considered in the dissertation \cite{vonP}, and he also proves that $K(G) \le Z(G)$ and $K(G/K(G)) = 1$  (see Definition 5.1.7 and Lemma 5.1.8 in that dissertation).  

The set $K(G)$ was considered for several special cases in \cite{BeBhu}.  They consider the cases of a group $G \times Z_n$ where ${\rm gcd} (|G|,n) = 1$ (Theorem 3.1 of \cite{BeBhu}), when $G$ is abelian (Theorem 3.2 of \cite{BeBhu}), and when $G$ is a non-abelian $p$-group for some prime $p$ (Theorem 3.3 of \cite{BeBhu}).   In addition, they showed that if $G$ is a nonabelian simple group, then $K(G) = 1$ (see Theorem 3.4 of \cite{BeBhu}).  The set $K(G)$ is considered for nilpotent groups in Proposition 4.5 of \cite{MaShe}.

We need a reformulation of a result in our previous paper (Theorem 1.4 of \cite{Zgps}).  We note that a result similar to Theorem 1.4 of \cite{Zgps} has been proved by Cameron in Theorem 9.1 (b) of \cite{camex} and by Mahmoudifar and Babai in the Main Theorem of \cite{MaBa}. We note that \cite{MaBa} was submitted in 2018, two years before \cite{Zgps} was submitted, even though it appeared after it.  Also, Cameron's paper was submitted in April 2021 which was before \cite{Zgps} appeared in print (although \cite{Zgps} did appear electronically in December 2020.)  In addition, we remark that Lemma 2.7 of \cite{BeKiMu} proves a version of this next lemma.  For completeness and to keep this paper somewhat self contained, we provide a proof of the reformulation for the reader's convenience.

\begin{lem} \label{equiv}
Let $G$ be a group and let $p$ be a prime.  Then $G$ has unique subgroup of order $p$ and that subgroup is central in $G$ if and only if $p$ divides $|K(G)|$.
\end{lem}

\begin{proof}
Assume that $G$ has a unique subgroup of order $p$, say $Z$, and $Z \le Z(G)$.
Of course, $Z$ is cyclic, and so we can write $Z = \langle z \rangle$ for some $z \in G$.
If $g \in G$ is a $p'$-element, then $\langle z, g\rangle$ is cyclic since $g$ and $z$ commute and have coprime orders.
If $g \in G$ has order divisible by $p$, then $\langle g\rangle$ has a subgroup of order $p$, and so $Z \le \langle g\rangle$ by the uniqueness of $Z$.
Hence, $\langle g, z\rangle = \langle g\rangle$ is cyclic.
We conclude that $z \in K(G)$.

Conversely, assume that $p$ divides $|K(G)|$.
As we mentioned previously, the set $K(G)$ is actually a subgroup of $G$, and so, under the hypothesis that $p$ divides $|K(G)|$, we know that $K(G)$ has a subgroup $Z$ of order $p$.
As $K(G) \le Z(G)$, the subgroup $Z$ is central.
Let $X$ be a subgroup of $G$ such that $|X| =p$.
Write $X =\langle x \rangle$ and $Z = \langle z\rangle$.
As $z \in K(G)$, we know that $\langle x, z\rangle$ is cyclic.
Hence, $\langle x, z\rangle$ has exactly one subgroup of order $p$.
Since $X$ and $Z$ are subgroups of $\langle x,z\rangle$ such that $|X| = |Z| = p$, we deduce that $X = Z$.
Since $X$ was an arbitrary subgroup of $G$ of order $p$, we conclude that $Z$ is the only subgroup of $G$ with order $p$.
\end{proof}

We close the discussion on the background of universal vertices by mentioning that \cite{BeDe} used Theorem 1.8 of \cite{Zgps} in Corollary 4.2 of \cite{BeDe} to explicitly determine $K(G)$ when $G$ is a nilpotent group.

We now work to characterize solvable groups $G$ with $K(G) > 1$.
In the following lemma, we shall write $\widetilde{\textnormal{GL}}_{2}(3)$ for the group that is isoclinic to $\textnormal{GL}_{2}(3)$ of order $48$ that has a generalized quaternion group of order $16$ as a Sylow $2$-subgroup.  We note that the solvability hypotheses on this lemma is necessary since, for example, $\textnormal{SL}_2 (5)$ would be a nonsolvable counterexample.

In Lemma \ref{generalizedqua}, we will use some structural properties of the automorphism group of a generalized quaternion group of order at least 16.
More specifically, we need the fact that if $Q$ is a generalized quaternion group of order at least 16, then $\textnormal{Aut}(Q)$ is a $2$-group.
This fact is mentioned briefly on page 321 in \cite{berk}, but we include the elementary proof here for completeness.
Of course, we shall use a few basic properties of a generalized quaternion group; these properties can be found in Theorem 4.1 and Theorem 4.2 in \cite{suzvolii}.

\begin{lem}
If $Q$ is a generalized quaternion group of order $2^n$, where $2^n \ge 16$, then $A = \textnormal{Aut}(Q)$ is a $2$-group.
\end{lem}

\begin{proof}
We know that $Q$ has a cyclic maximal subgroup $M$ such that $|Q : M| = 2$.
Clearly $Q' \le M$, and so $M \le C_Q(Q')$.
As $Q$ has nilpotence class $n-1 \ge 3$, we know that $Q'$ is noncentral, and so $C_Q(Q') = M$ by the maximality of $M$.
The centralizer of a characteristic subgroup is characteristic, and so $M$ is a characteristic subgroup of $Q$.

For each $a \in A$, let $\hat{a}$ be the automorphism of $M$ obtained by restricting $a$ to $M$.
(This restriction makes sense as $M$ is characteristic in $Q$.)
The map $a \mapsto \hat{a}$ is homomorphism from $A$ into $\textnormal{Aut}(M)$ with kernel $C_A(M)$.
Hence, $A/C_A(M)$ embeds into $\textnormal{Aut}(M)$ by the $N/C$-Theorem.
The automorphism group of a cyclic $2$-group is a $2$-group, and so $A/C_A(M)$ is a $2$-group.

Let $p$ be a prime divisor of $|C_A(M)|$, and let $a \in C_A(M)$ with $o(a) = p$.
As $a \in C_A(M)$, we know that $M \le C_Q(a)$.
If $C_Q(a) = Q$, then $a = 1$, contrary to our choice of $a$.
Hence, $M = C_Q(a)$.
If $\varOmega = Q \setminus M$, then $a$ permutes the elements of $\varOmega$.
In fact, $a$ has no fixed points on $\varOmega$, as $M = C_Q(a)$.
So, $a$ partitions $\varOmega$ into orbits of length $p$.
As $|\varOmega| = 2^n - 2^{n-1} = 2^{n-1}$, the prime $p$ must divide $2^{n-1}$, which forces $p =2$.
We deduce that $C_A(M)$ is also a $2$-group.
Indeed, $A$ is a $2$-group.
\end{proof}

We now prove a critical lemma concerning solvable groups that have a generalized quaternion Sylow $2$-subgroup.

\begin{lem}\label{generalizedqua}
If $G$ is a solvable group and a Sylow $2$-subgroup of $G$ is generalized quaternion, then one of the following situations occurs:
\begin{enumerate}
\item $G$ has a normal $2$-complement
\item $G/O_{2'}(G)\cong\textnormal{SL}_{2}(3)$
\item $G/O_{2'}(G)\cong\widetilde{\textnormal{GL}}_{2} (3)$.
\end{enumerate}
\end{lem}

\begin{proof}
We start by making a reduction.
Let $U = O_{2'}(G)$.
Since $G$ is solvable, the factor group $G/U$ is solvable.  If $T$ is a Sylow $2$-subgroup of $G$, then $TU/U\cong T$ is a Sylow $2$-subgroup of $G/U$; thus, a Sylow $2$-subgroup of $G/U$ is generalized quaternion.  In particular, $G/U$ satisfies the same hypotheses as $G$.  As $O_{2'}(G/U)$ is trivial, we may assume without loss of generality that $U= 1$.  Note that under the assumption that $U$ is trivial, (1) occurs if and only if $G$ is a $2$-group. 
	
Set $Q=O_2(G)$.	Since $Q$ is a subgroup of a generalized quaternion group, $Q$ has a unique subgroup of order $2$ and is therefore either cyclic or generalized quaternion.  If $Q$ is cyclic, then $\aut{Q}$ is a $2$-group.  If $Q$ is generalized quaternion of order at least $16$, then $\aut{Q}$ is a $2$-group.  By the $N/C$-Theorem, $G/C_G(Q)$ embeds in $\aut{Q}$.  So, if $Q$ is either cyclic or generalized quaternion of order at least $16$, then $G/C_G(Q)$ is a $2$-group.  The group $G$ is solvable and $U=1$; thus, $C_G(Q)\le Q$ by the Hall-Higman Lemma 1.2.3 (see \cite{isaacs}, Theorem 3.21). Hence, $G$ is a $2$-group and we have conclusion (1).
	
Assume now that $Q$ is quaternion of order $8$.  As mentioned in the previous paragraph, the Hall-Higman Lemma 1.2.3 applies and tells us that $C_G(Q)\le Q$.  In particular, $C_G(Q)=Z(Q)$ and so $|C_G(Q)|=2$. 	We note that $Q/C_G(Q)$ is isomorphic to the Klein $4$-group.  If $G=Q$, then we are done.  Assume that $Q<G$.  The section $G/C_G(Q)$ embeds in $\aut{Q}$ by the $N/C$-Theorem.  As $Q$ is quaternion of order $8$, $\aut{Q}\cong S_{4}$.  The only subgroups of $S_{4}$ that properly contain a subgroup isomorphic to the Klein $4$-group are $A_{4}$ and $S_{4}$.  Thus, $G/C_G(Q)\cong A_{4}$ or $G/C_G(Q)\cong S_{4}$.  If $G/C_G(Q)\cong A_{4}$, then $G\cong\textnormal{SL}_{2}(3)$.  If $G/C_G(Q)\cong S_{4}$, then since the only central extension of $S_{4}$ that has a generalized quaternion Sylow $2$-subgroup is $\widetilde{\textnormal{GL}}_{2} (3)$, we conclude that $G\cong\widetilde{\textnormal{GL}}_{2} (3)$.
\end{proof}

We now are ready to prove Theorem \ref{K(G) main}.

\begin{proof}[Proof of Theorem \ref{K(G) main}]
Suppose that $p$ divides $|K(G)|$. If $p$ is odd, then, using Lemma \ref{equiv}, we immediately deduce that we have conclusion (1). So assume that $p=2$. A Sylow $2$-subgroup of $G$ is either cyclic or generalized quaternion. If a Sylow $2$-subgroup of $G$ is cyclic, then we again obtain conclusion (1). Assume that a Sylow $2$-subgroup of $G$ is generalized quaternion. The subgroup $K(G)$ is cyclic, and so if $|K(G)|_2 > 2$, then $K(G)$ has an element of order $4$, say $u$. Now $U= \langle u \rangle$ is a normal $2$-subgroup of $G$ and so $U \le P$. Since $U \le K(G) \le Z(G)$, we have that $U \le Z(P)$. But $|Z(P)| = 2$, and so we have a contradiction. We conclude that $|K(G)|_2 =2$. In particular, $K(G) \cap P = Z(P)$.  If $G$ has a normal $2$-complement $Q$, then we have (2a) as $Z(P)\le C_P(Q)$.  Assume that $G$ does not have a normal $2$-complement.  Using Lemma \ref{generalizedqua}, we have that $G/O_{2'}(G)$ is isomorphic to either $\textnormal{SL}_2(3)$ or $\widetilde{\textnormal{GL}}_{2}(3)$.  If $3$ divides $|K(G)|$, then $G$ has a cyclic Sylow $3$-subgroup and $G$ is $3$-nilpotent by Lemma \ref{equiv} and Corollary 5.30 in \cite{isaacs}.  Because $\widetilde{\textnormal{GL}}_{2}(3)$ does not have a normal Sylow $2$-subgroup, $G/O_3(G)\cong\textnormal{SL}_{2}(3)$.

If (1) occurs, then $p$ divides $|K(G)|$ by Lemma \ref{equiv}. Assume (2). Let $U = O_{2'}(G)$ and let $Z = Z(P)$. Because $G/U \cong X \in \{ \textnormal{SL}_2(3), \widetilde{\textnormal{GL}}_{2}(3)\}$, we deduce that $ZU/U = Z(G/U)$, and so $ZU/U$ is normal in $G/U$. By the Correspondence Theorem (\cite{isaacs}, X.21), $ZU$ is normal in $G$. As $Z$ is a Sylow $2$-subgroup of $ZU$, we have that $G = ZU N_G(Z) = UN_G(Z)$ by the Frattini Argument (see Lemma 1.13 in \cite{isaacs}). By our hypothesis, $Z \le C_P(U)$, and so $U \le C_G(Z)$. Hence, $G = N_G(Z)$ and $Z$ is normal in $G$. Since $Z$ is a normal subgroup of $G$ such that $|Z| = 2$, we have that $Z \le Z(G)$.
If $T$ is a subgroup of $G$ such that $|T| = 2$, then $\langle Z ,T \rangle = ZT$ is a $2$-subgroup of $G$.
Because the Sylow 2-subgroups of $G$ are generalized quaternion, we deduce that $T= Z$.
We therefore appeal to Lemma \ref{equiv} and conclude that $2$ divides $|K(G)|$.
\end{proof}

\section{$\{p,q\}$-groups}

Let $p$ and $q$ be distinct prime numbers, and let $G$ be a $\{p,q\}$-group.
In this section, we present a necessary and sufficient condition for $\textnormal{diam}(\Delta(G))=2$.
It turns out that $\textnormal{diam}(\Delta(G))=2$ precisely when a universal vertex exists.
We use the notation $\sim$ to denote an edge between $x$ and $y$ in $\Delta(G)$, and $G^{\#} = G \setminus \{ 1\}$.

We start with an easy lemma concerning elements of prime order in a $\{p,q\}$-group $G$ with $\textnormal{diam}(\Delta(G))=2$.

\begin{lem}\label{elementsofprimeorderinapqgroup}
Let $p$ and $q$ be distinct primes.
If $G$ is a $\{p,q\}$-group such that $\textnormal{diam}(\Delta(G))=2$, then $x\sim y$ for every element $x$ of order $p$ and for every element $y$ of order $q$.
\end{lem}

\begin{proof}
Let $x\in G^{\#}$ be an arbitrary element of order $p$ and let $y\in G^{\#}$ be an arbitrary element of order $q$.  If $x \sim y$, then the result holds.  
Since $\textnormal{diam}(\Delta(G))=2$, there exists some $z\in G^{\#}$ such that $x\sim z\sim y$.
Of course, at least one of primes $p,q$ divides $o(z)$.
Without loss of generality, we may assume that $p$ divides $o(z)$.
Since $\langle x, z\rangle$ is cyclic, $\langle x\rangle=\langle z^{n}\rangle$ for some integer $n$.
Now, $x,y\in\langle z,y\rangle$.
As $\langle z,y\rangle$ is cyclic, we conclude that $x\sim y$.
\end{proof}

We now prove our main result.
Notice that \textit{sufficiency} is essentially a group-theoretic formulation of the existence of a universal vertex in the corresponding enhanced power graph.  Hence, this result yields Theorem \ref{Main 2}.

\begin{thm}\label{pqdiameter2}
Let $p$ and $q$ be distinct primes, and let $G$ be a $\{p, q\}$-group.
Then, $\textnormal{diam}(\Delta(G))=2$ if and only if $G$ has a unique subgroup of order $p$ or a unique subgroup of order $q$ and that subgroup is central in $G$.
\end{thm} 

\begin{proof}
Assume that the group $G$ has a unique subgroup $Z$ of prime order and that $Z\le Z(G)$.
Write $Z=\langle z\rangle$ and, without loss of generality, we may assume that $|Z|=p$.
Now, let $g\in G^{\#}$.
If $p$ does not divide $o(g)$, then $g\sim z$ since commuting elements with coprime orders generate a cyclic subgroup.
If $p$ does divides $o(g)$, then, since $Z$ is the unique subgroup of order $p$, we have that $\langle g^{n}\rangle=Z$ for some integer $n$.
So, again, $g\sim z$.
Hence, $z$ is a universal vertex in $\Delta(G)$, which yields the bound $\textnormal{diam}(\Delta(G))\le 2$.
As $G$ is non-cyclic, we conclude that $\textnormal{diam}(\Delta(G))=2$.

Next, assume that $\textnormal{diam}(\Delta(G))=2$.
We start by establishing the existence of a unique subgroup of prime order.
If $G$ has a unique subgroup of order $p$, then we are done.
So, suppose that $P_{1}=\langle x_{1}\rangle$ and $P_{2}=\langle x_{2}\rangle$ are distinct subgroups of order $p$.
Note that $d(x_{1},x_{2})=2$.
Write $x_{1}\sim y\sim x_{2}$ for $y\in G^{\#}$.
Replacing $y$ if necessary, we may assume that $o(y)=q$.
Let $w$ be an element of order $q$.
Consider the products $x_{1}y$ and $x_{2}w$.
The elements $x_{1}$ and $y$ have coprime orders.
The elements $x_{2}$ and $w$ have coprime orders.
These elements are commuting pairs of elements by Lemma \ref{elementsofprimeorderinapqgroup}.
So, $\langle x_{1}, y\rangle=\langle x_{1}y\rangle$ and $\langle x_{2},w\rangle=\langle x_{2}w\rangle$.
Now, if $x_{1}y\sim x_{2}w$, then $\langle x_{1}\rangle =\langle x_{2} \rangle$, a contradiction.
Hence, $d(x_{1}y,x_{2}w)=2$.
Write $x_{1}y\sim z\sim x_{2}w$ for $z\in G^{\#}$.
Observe that this situation forces $z$ to be a $q$-element, and so we may assume that $o(z)=q$.
Now, $\langle y\rangle=\langle z\rangle=\langle w\rangle$.
In particular, $\langle y\rangle$ is the unique subgroup of $G$ of order $q$.

In the previous paragraph, we established that $G$ has a unique subgroup of prime order.
Without loss of generality, we proceed by assuming that $G$ has a unique subgroup of order $p$, say $P=\langle x\rangle$.
If $P\le Z(G)$, then we are done.
Assume that $P\not\le Z(G)$, and so there exists some $g\in G$ such that $P\not\le C_G(g)$.
If $p$ divides $o(g)$, then $P=\langle x\rangle\le\langle g\rangle$ since $P$ is the \textit{unique} subgroup of order $p$.
But then, $P\le C_G(g)$, a contradiction.
Hence, $g$ must be a $q$-element.
As we proved in Lemma \ref{elementsofprimeorderinapqgroup}, an element of order $p$ is adjacent to an element of order $q$ in $\Delta(G)$.
Hence, $o(g)=q^{l}$, where $l\ge 2$.
For some positive integer $n$, the subgroup $Q=\langle g^{n} \rangle$ has order $q$.
We claim that $Q$ is the \textit{unique} subgroup of $G$ with order $q$ and that $Q\le Z(G)$.
Let $Q_{0}=\langle y\rangle$ be a subgroup of order $q$.
Note that $\langle x,y\rangle=\langle xy\rangle$ and, by our hypothesis, $d(xy,g)\le 2$.
If $xy\sim g$, then $P=\langle x\rangle\le C_G(g)$, a contradiction.
Hence, $d(xy,g)=2$.
Write $xy\sim h\sim g$ for $h\in G^{\#}$.
If $p$ divides $o(h)$, then $\langle x\rangle =\langle h^{t}\rangle\le C_G(g)$ for some integer $t$.
But then we have arrived at another contradiction.
The element $h$ is therefore a $q$-element, and so $Q_{0}=\langle y\rangle=\langle h^{s}\rangle=\langle g^{n}\rangle=Q$ for some positive integer $s$.
Since $Q_{0}$ was an arbitrary subgroup of $G$ of order $q$, we conclude that $Q$ is the unique subgroup of $G$ with order $q$, as wanted.
Finally, observe that the section $G/C_G(P)$ embeds in $\aut{P}$ and the section $G/C_G(Q)$ embeds in $\aut{Q}$.
By arithmetical considerations, one of these sections is trivial.
Since $P$ is non-central, we must have that $G=C_G(Q)$, and so $Q\le Z(G)$.
\end{proof}

We now prove Theorem \ref{Main 2}.

\begin{proof}[Proof of Theorem \ref{Main 2}]
Assume that $\textnormal{diam}(\Delta(G)) =2$.
By Theorem \ref{pqdiameter2}, we know that $G$ either has a unique subgroup of order $p$ or a unique subgroup of order $q$ and that subgroup is central.
By Lemma \ref{equiv}, we know that either $p$ or $q$ divides $|K (G)|$, and so $K(G) > 1$.

Conversely, suppose $K(G) > 1$.  Then either $p$ or $q$ divides $|K(G)|$.  In light of Lemma \ref{equiv}, we see that $G$ has either a unique subgroup of order $p$ or and a unique subgroup of order $q$ that is central.  Applying Theorem \ref{pqdiameter2}, we see that $\Delta (G)$ has diameter $2$, as desired.
\end{proof}

\end{document}